\documentclass[11pt,oneside,reqno]{amsart}
\usepackage[utf8]{inputenc}
\usepackage[T1]{fontenc}
\usepackage{lmodern}
\usepackage{amssymb, amsmath, amsthm, mathtools, cases,enumerate}
\usepackage{color}
\usepackage{hyperref,mathscinet}
\usepackage{natbib}
\usepackage{microtype}
\usepackage{hyphenat}

\newcommand{\TITLE}{Different degrees of non-compactness for optimal Sobolev embeddings}

\hypersetup{
	unicode=true,
	breaklinks=true,
	pdftitle={\TITLE},
	pdfauthor={Jan Lang and Zden\v ek Mihula}
	}

\title{\TITLE}
\author{Jan Lang and Zden\v ek Mihula}

\address{Jan Lang, Department of Mathematics, The Ohio State University, 231 West 18th Avenue, Columbus, OH 43210-1174}
\email{lang.162@osu.edu}
\urladdr{\href{https://orcid.org/0000-0003-1582-7273}{0000-0003-1582-7273}}

\address{Zden\v ek Mihula, Czech Technical University in Prague, Faculty of Electrical Engineering, Department of Mathematics, Technick\'a~2, 166~27 Praha~6, Czech Republic}
\email{mihulzde@fel.cvut.cz}
\urladdr{\href{https://orcid.org/0000-0001-6962-7635}{0000-0001-6962-7635}}

\date{\today}

\theoremstyle{plain}
\newtheorem{theorem}{Theorem}[section]

\newtheorem{proposition}[theorem]{Proposition}
\newtheorem{lemma}[theorem]{Lemma}

\theoremstyle{definition}
\newtheorem{remark}[theorem]{Remark}

\newcommand{\R}{\mathbb R}
\newcommand{\N}{\mathbb N}
\newcommand{\rd}{\R^d}
\newcommand{\Pol}{\mathcal P}

\newcommand{\dx}{{\fam0 d}}
\renewcommand{\d}[1]{\,\dx #1}

\DeclareMathOperator{\spt}{supp}
\DeclareMathOperator{\diam}{diam}
\DeclareMathOperator{\rank}{rank}

\newcommand{\myref}[2]{\hyperref[#2]{#1~\ref*{#2}}}

\numberwithin{equation}{section}

\subjclass[2020]{46E35, 47B06, 46B50}
\keywords{Sobolev spaces, optimal spaces, compactness, Bernstein numbers, singular operators}
\thanks{This research was partly supported by the project OPVVV CAAS CZ.02.1.01/0.0/0.0/16\_019/0000778, and by the grant P201/21-01976S of the Czech Science Foundation.}

\begin{document}
\setcitestyle{numbers}
\bibliographystyle{plainnat}

\begin{abstract}
The structure of non-compactness of optimal Sobolev embeddings of $m$-th order into the class of Lebesgue spaces and into that of all rearrangement-invariant function spaces is quantitatively studied. Sharp two-sided estimates of Bernstein numbers of such embeddings are obtained. It is shown that, whereas the optimal Sobolev embedding within the class of Lebesgue spaces is finitely strictly singular, the optimal Sobolev embedding in the class of all rearrangement-invariant function spaces is not even strictly singular.
\end{abstract}
\maketitle


\section{Introduction}

Sobolev spaces and their embeddings into Lebesgue or Lorentz spaces (on an open set $\Omega \subseteq {\mathbb{R}}^d$) keep a prominent position in the theory of partial differential equations, and any information about structure of such embeddings is far-reaching.

There is a vast amount of literature devoted to study of conditions under which Sobolev embeddings are compact. Quality of compactness is often studied by the speed of decay of different $s$-numbers, which is connected to spectral theory of corresponding differential operators and provides estimates of the growth of their eigenvalues (see \cite{ET:96}). However, much less literature is devoted to study of the structure of non-compact Sobolev embeddings, which is related to the shape of essential spectrum (see \cite{EE:87}).

There are three common ways under which Sobolev embeddings can become non-compact:
\begin{enumerate}[(i)]
\item when the underlying domain is unbounded (see~\cite{AFbook}, cf.~\cite{EML:21});
\item when the boundary $\partial\Omega$ of $\Omega$ is too irregular (see~\cite{KMR:01, LM:08, Mabook, MP:97});
\item when the target space is optimal or ``almost optimal'' (see~\cite{KP:08, LMP:22} and references therein).
\end{enumerate}

In this paper we will focus on the third case. We will obtain new information about the structure of non-compactness of two optimal Sobolev embeddings\textemdash namely
\begin{align}
I\colon V_0^{m, p}(\Omega) \to L^{p^*}(\Omega) \label{prel:optimal_Leb}
\intertext{and}
I\colon V_0^{m, p}(\Omega) \to L^{p^*, p}(\Omega), \label{prel:optimal_Lor}
\end{align}
where $1\leq m < d$, $p\in[1, d/m)$ and $p^*=dp/(d - mp)$. Here $\Omega$ is a bounded open set, and the subscript $0$ means that the (ir)regularity of $\Omega$ is immaterial (see~Section~\ref{sec:prel} for precise definitions). The target spaces in both embeddings are in a sense optimal. The Lebesgue space $L^{p^*}(\Omega)$ is well known to be the optimal target space in \eqref{prel:optimal_Leb} among all Lebesgue spaces\textemdash that is, $L^{p^*}(\Omega)$ is the smallest Lebesgue space $L^q(\Omega)$ such that $I\colon V_0^{m, p}(\Omega) \to L^{q}(\Omega)$ is valid. However, it is also well known (\cite{P:66}) that \eqref{prel:optimal_Leb} can be improved to \eqref{prel:optimal_Lor} if one allows not only Lebesgue spaces but also Lorentz spaces, which form a richer class of function spaces. Since $L^{p^*, p}(\Omega) \subsetneq L^{p^*}(\Omega)$, the latter is indeed an improvement. Furthermore, the Lorentz space $L^{p^*, p}(\Omega)$ is actually the optimal target space in \eqref{prel:optimal_Lor} among all rearrangement-invariant function spaces (see~\cite{KP:06})\textemdash that is, if $Y(\Omega)$ is a rearrangement-invariant function space (e.g., a Lebesgue space, a Lorentz space, or an Orlicz space, to name a few customary examples) such that $I\colon V_0^{m, p}(\Omega) \to Y(\Omega)$ is valid, then $L^{p^*, p}(\Omega) \subseteq Y(\Omega)$.

Not only are both embeddings \eqref{prel:optimal_Leb} and \eqref{prel:optimal_Lor} non-compact, but they are also in a sense ``maximally non-compact'' as their measures of non-compactness (in the sense of \cite[Definition~2.7]{EE:87}) are equal to their norms. This was proved in \cite{B:20, H:03}. Moreover, even when $L^{p^*}(\Omega)$ is enlarged to the weak Lebesgue space $L^{p^*, \infty}(\Omega)$, which satisfies $L^{p^*, p}(\Omega) \subsetneq L^{p^*}(\Omega) \subsetneq L^{p^*, \infty}(\Omega)$, the resulting Sobolev embedding is still maximally non-compact. This was proved in \cite{LMOP:21}. These results may suggest that the ``quality'' and the structure of these non-compact embeddings should be the same.

However, there are other possible points of view on the quality of non-compactness. One of them is the question of whether a non-compact Sobolev embedding is strictly singular or even finitely strictly singular. Strictly singular operators and finitely strictly singular ones are important classes of operators as spectral properties of such operators are very close to those of compact ones. In this regard, it follows from \cite{BG:89} that the Sobolev embedding $I\colon V_0^{1,1}(\Omega) \to L^{d/(d - 1)}(\Omega)$, which is a particular case of \eqref{prel:optimal_Leb} with $m = p = 1$, is finitely strictly singular. Furthermore, it was also shown there that the almost optimal critical Sobolev embedding $I\colon V_0^{d, 1}((0,1)^d) \to L^{\infty}((0,1)^d)$ is finitely strictly singular, too. Finally, the same was proved in \cite{LM:19} for the optimal first-order Sobolev embedding into the space of continuous functions on a cube. These results suggest a hypothesis that non-compact Sobolev embeddings could be finitely strictly singular or at least strictly singular.

In this paper we will show that this hypothesis is correct for the ``almost optimal'' Sobolev embedding \eqref{prel:optimal_Leb}, but it is wrong for the ``really optimal'' Sobolev embedding \eqref{prel:optimal_Lor}. In other words, \eqref{prel:optimal_Lor} is an example of a Sobolev embedding whose target space is optimal among all rearrangement-invariant function spaces that is not a singular map (i.e., there exists an infinite dimensional subspace on which the embedding is invertible), but if the target space is slightly enlarged to an ``almost optimal'' one (i.e., the target space is optimal only in the smaller class of Lebesgue spaces), then the resulting Sobolev embedding \eqref{prel:optimal_Leb} is finitely strictly singular (i.e., its Bernstein numbers are decaying to zero). In the case of \eqref{prel:optimal_Leb}, we prove a two-sided estimate of the Bernstein numbers corresponding to the embedding\textemdash the estimate is sharp up to multiplicative constants. In the case of \eqref{prel:optimal_Lor}, we show that all its Bernstein numbers coincide with the norm of the embedding.

The paper is structured as follows. In the next section, we recall definitions and notation used in this paper, as well as some background results. In \myref{Section}{sec:main}, we start with a couple of auxiliary results, which may be of independent interest, then we focus on the ``almost optimal'' embedding (\myref{Theorem}{thm:optimal_Sobolev_Lebesgue}), and finally on the ``really optimal'' one (\myref{Theorem}{thm:optimal_Sobolev_Lorentz}).


\section{Preliminaries} \label{sec:prel}

Here we establish the notation used in this paper, and recall some basic definitions and auxiliary results.

Any rule $s\colon T\to\left\{s_n(T)\right\}_{n=1}^\infty$ that assigns each bounded linear operator $T$ from a Banach space $X$ to a Banach space $Y$ (we shall write $T\in B(X,Y)$) a sequence $\left\{s_n(T)\right\}_{n=1}^\infty$ of nonnegative numbers having, for every $n\in\N$, the following properties:
\begin{itemize}
\item[(S1)] $\|T\|=s_1(T)\geq s_2(T)\geq\cdots\geq0$;
\item[(S2)] $s_n(S+T)\leq s_n(S)+\|T\|$ for every $S\in B(X,Y)$;
\item[(S3)] $s_n(BTA)\leq\|B\|s_n(T)\|A\|$ for every $A\in B(W,X)$ and $B\in B(Y,Z)$, where $W,Z$ are Banach spaces;
\item[(S4)] $s_n(I\colon E\to E)=1$ for every Banach space $E$ with $\dim E\geq n$;
\item[(S5)] $s_n(T)=0$ if $\rank T< n$;
\end{itemize}
is called a \emph{strict $s$-number}. Notable examples of strict $s$-numbers are \emph{the approximation numbers} $a_n$, \emph{the Bernstein numbers} $b_n$, \emph{the Gelfand numbers} $c_n$, \emph{the Kolmogorov numbers} $d_n$, \emph{the isomorphism numbers} $i_n$, or \emph{the Mityagin numbers} $m_n$. For their definitions and the difference between strict $s$-numbers and `non-strict' $s$-numbers, we refer the reader to \citep[Chapter~5]{EL:11} and references therein. We say that a (strict) $s$-number is \emph{injective} if the values of $s_n(T)$ do not depend on the codomain of $T$. More precisely, $s_n(J_N^Y \circ T) = s_n(T)$ for every closed subspace $N\subseteq Y$ and every $T\in B(X, N)$, where $J_N^Y\colon N \to Y$ is the canonical embedding operator.

In this paper, we will only need the definition of the Bernstein numbers. The $n$-th Bernstein number $b_n(T)$ of $T\in B(X,Y)$ is defined as
\begin{equation*}
b_n(T)=\sup_{X_n \subseteq X} \inf_{\substack{x\in X_n\\ \|x\|_{X} = 1}}\|Tx\|_{Y},
\end{equation*}
where the supremum extends over all $n$-dimensional subspaces of $X$. The Bernstein numbers are the smallest injective strict $s$-numbers (\citep[Theorem~4.6]{P:74}), that is,
\begin{equation}\label{prel:bernstein_numbers_smallest_injective}
b_n(T) \leq s_n(T)
\end{equation}
for every injective strict $s$-number $s$, for every $T\in B(X,Y)$, and for every $n\in\N$.

An operator $T\in B(X,Y)$ is said to be \emph{strictly singular} if there is no infinite
dimensional closed subspace $Z$ of $X$ such that the restriction $T\rvert_{Z}$ of
$T$ to $Z$ is an isomorphism of ${Z}$ onto $T(Z).$ Equivalently, for
each infinite dimensional (closed) subspace $Z$ of $X$,
\begin{equation*}
\inf\left\{  \left\Vert Tx\right\Vert _{Y}\colon \left\Vert x\right\Vert _{X}=1,x\in
Z\right\}  = 0.
\end{equation*}
An operator $T\in B(X,Y)$ is said to be \emph{finitely strictly singular} if it has the property that given any $\varepsilon>0$ there exists
$N(\varepsilon)\in\mathbb{N}$ such that if $E$ is a subspace of $X$ with $\dim
E\geq N(\varepsilon),$ then there exists $x\in E$, $\left\Vert
x\right\Vert _{X}=1$, such that $\left\Vert Tx\right\Vert _{Y}\leq
\varepsilon$. This can be expressed in terms of the Bernstein numbers of $T$. The operator $T$ is finitely strictly singular if and only if
\begin{equation*}
\lim_{n \to \infty} b_n(T) = 0.
\end{equation*}
The relations between these two notions and that of compactness of $T$ are illustrated by the following diagram:
\begin{equation*}
\text{$T$ is compact }\Longrightarrow\text{ }\text{$T$ is finitely strictly singular
}\Longrightarrow\text{ }\text{$T$ is strictly singular};
\end{equation*}
moreover, each reverse implication is false in general. For further details and
general background information concerning these matters we refer the interested reader to
\cite{AK:16}, \cite{LR-P:14} and \cite{P:07}.

Throughout the rest of this section, $X$ denotes a Banach space. The operator norm of the projection $Q\colon L^2([0,1], X) \to L^2([0,1], X)$ defined as
\begin{equation*}
Qf(t) = \sum_{j = 1}^\infty \left( \int_0^1 f(s) r_j(s) \d{s} \right) r_j(t),\ f\in L^2([0,1], X),
\end{equation*}
is called the \emph{$K$-convexity constant} of $X$. Here $\{r_j\}_{j = 1}^\infty$ are the Rademacher functions. The $K$-convexity constant of $X$ is denoted by $K(X)$. If $\dim X = n$, then (e.g., see~\cite[Theorem~6.2.4]{A-AGM:15})
\begin{equation}\label{prel:pisier_inequality}
K(X) \leq c \log(1 + d(X, \ell_2^n)).
\end{equation}
Here $c$ is an absolute constant and $d(X, \ell_2^n)$ is the Banach--Mazur distance, that is,
\begin{equation*}
d(X, \ell_2^n) = \inf\{\|T\|\|T^{-1}\|\colon \text{$T$ is a linear isomorphism of $X$ onto $\ell_2^n$}\}.
\end{equation*}

We say that $X$ is of \emph{cotype $2$} if there is a constant $\gamma$ such that
\begin{equation*}
\left(\sum_{j = 1}^m \|x_j\|_{X}^2 \right)^{1/2} \leq \gamma \int_0^1 \Big \|\sum_{j = 1}^m x_j r_j(t) \Big\|_X \d{t}
\end{equation*}
for every $\{x_j\}_{j=1}^m\subseteq X$, $m\in \N$. We denote the least such a $\gamma$ by $C_2(X)$.

Let $Y$ be a Banach space such that $X \subseteq Y$. We say that the inclusion is \emph{$2$-absolutely summable} if there is a constant $\gamma$ such that
\begin{equation*}
\left( \sum_{j = 1}^m \|x_j\|_{Y}^2 \right)^{1/2} \leq \gamma \sup\left\{ \left( \sum_{j = 1}^m |x^*(x_j)|^2 \right)^{1/2}\colon \|x^*\|_{X^*} \leq 1 \right\}
\end{equation*}
for every $\{x_j\}_{j=1}^m\subseteq X$, $m\in \N$. We denote the least such a $\gamma$ by $\pi_2(X\hookrightarrow Y)$.

Let $A,B$ be subsets of $X$. We denote the minimum number of points $x_1,\dots, x_m\in X$ such that
\begin{equation}\label{prel:entropy_def}
A \subseteq \bigcup_{j=1}^m (x_j + B)
\end{equation}
by $E(A, B)$. In general, it may happen that $E(A, B) = \infty$, but in our case it will always be a finite number. $\bar{E}(A, B)$ denotes the minimum number of points $x_1,\dots, x_m\in A$ such that \eqref{prel:entropy_def} holds.

Let $(R, \mu)$ be a nonatomic measure space and $p\in[1, \infty)$. As usual, $L^p(R, \mu)$ denotes the \emph{Lebesgue space} endowed with the norm
\begin{equation*}
\|f\|_{L^p(R, \mu)} = \left( \int_R |f|^p \d{\mu} \right)^{\frac1{p}},\ f\in L^p(R, \mu).
\end{equation*}
Let $q\in[1, p]$. The \emph{Lorentz space} $L^{p,q}(R, \mu)$ is the Banach space of all $\mu$-measurable functions $f$ in $R$ for which the functional
\begin{equation*}
\|f\|_{L^{p,q}(R, \mu)} = \left(\int_0^\infty t^{\frac{q}{p} - 1} f^*(t)^q \d{t} \right)^{\frac1{q}}
\end{equation*}
is finite\textemdash the norm on $L^{p,q}(R, \mu)$ is given by the functional. The function $f^*\colon (0, \infty) \to [0, \infty]$ is the (right-continuous) \emph{nonincreasing rearrangement} of $f$, that is,
\begin{equation*}
f^*(t) = \inf\{\lambda > 0\colon \mu(\{x\in R\colon |f(x)| > \lambda \})\leq t \},\ t\in(0, \infty).
\end{equation*}
Note that $f^*(t) = 0$ for every $t\in[\mu(R), \infty)$. Furthermore, we have (see~\cite[Chapter~2, Proposition~1.8]{BS})
\begin{equation*}
\|\cdot\|_{L^{p,p}(R, \mu)} = \|\cdot\|_{L^{p}(R, \mu)}.
\end{equation*}
When $R\subseteq\rd$ and $\mu$ is the $d$-dimensional Lebesgue measure, we write $L^p(R)$ and $L^{p,q}(R)$ instead of $L^p(R, \mu)$ and $L^{p,q}(R, \mu)$, respectively, and $|R|$ instead of $\mu(R)$ for short. We refer the interested reader to \cite[Chapter~8]{PKJF:13} for more information on Lorentz spaces. Assume that $(R, \mu)$ is probabilistic. We denote by $L^{\psi_2}(R,\mu)$ the \emph{Orlicz space} generated by the Young function
\begin{equation*}
\psi_2(t) = \exp(t^2) - 1,\ t\in[0,\infty).
\end{equation*}
The norm on $L^{\psi_2}(R,\mu)$ is given by
\begin{equation*}
\|f\|_{L^{\psi_2}(R,\mu)} = \inf\left\{ \lambda > 0\colon \int_R \psi_2\left(\frac{|f(x)|}{\lambda}\right) \d{\mu(x)} \leq 1 \right\}.
\end{equation*}
We have (e.g., see \cite[Lemma~3.5.5]{A-AGM:15})
\begin{equation}\label{prel:expL2_extrapolation_norm}
c\sup_{p\in[1, \infty)}\frac{\|f\|_{L^p(R,\mu)}}{\sqrt{p}} \leq \|f\|_{L^{\psi_2}(R,\mu)} \leq \tilde{c}\sup_{p\in[1, \infty)}\frac{\|f\|_{L^p(R,\mu)}}{\sqrt{p}}
\end{equation}
for every $f\in L^{\psi_2}(R,\mu)$. Here $c$ and $\tilde{c}$ are absolute constants. In particular, $L^{\psi_2}(R,\mu)$ is continuously embedded in $L^p(R,\mu)$ for every $p\in[1, \infty)$.

Throughout the entire paper, we assume that $d\in\N$, $d\geq2$. Let $G\subseteq\rd$ be a nonempty bounded open set. For $m\in\N$ and $p\in[1, \infty)$, $V^{m, p}(G)$ denotes the vector space of all $m$-times weakly differentiable functions in $G$ whose $m$-th order weak derivatives belong to $L^p(G)$. By $V_0^{m,p}(G)$ we denote the Banach space of all functions from $V^{m,p}(G)$ whose continuation by $0$ outside $G$ is $m$-times weakly differentiable in $\rd$ equipped with the norm  $\|u\|_{V_0^{m, p}(G)} = \| |\nabla^m u|_{\ell_p} \|_{L^p(G)}$. By $\nabla^m$ we denote the vector of all $m$-th order weak derivatives. When $G$ is regular enough (for example, Lipschitz), $V_0^{m, p}(G)$ coincides with the usual Sobolev space $W_0^{m, p}(G)$, up to equivalent norms.


\section{Different degrees of noncompactness}\label{sec:main}
An important property of both optimal Sobolev embeddings \eqref{prel:optimal_Leb} and \eqref{prel:optimal_Lor}, which we will exploit in both cases, is that their norms are homothetic invariant.

\begin{proposition}\label{prop:scaling_invariance}
Let $\Omega\subseteq\rd$ be a nonempty bounded open set, $m\in\N$, $1\leq m < d$, and $p\in[1, d/m)$. Let $p^* = dp/(d-mp)$ and $q\in[p, p^*]$. Denote by $I$ the identity operator $I\colon V_0^{m,p}(\Omega) \to L^{p^*, q}(\Omega)$. For every $0 < \lambda < \|I\|$ and every $\varepsilon > 0$, there exist a system of functions $\{u_j\}_{j=1}^\infty\subseteq V_0^{m,p}(\Omega)$ and a system of open balls $\{B_{r_j}(x_j)\}_{j=1}^\infty \subseteq \Omega$ with the following properties.
\begin{enumerate}[(i)]
	\item The balls $\{B_{r_j}(x_j)\}_{j=1}^\infty$ are pairwise disjoint.
	\item $\|u_j\|_{V_0^{m,p}(\Omega)} = 1$ and $\|u_j\|_{L^{p^*, q}(\Omega)} = \lambda$ for every $j\in\N$.
	\item $\spt u_j \subseteq B_{r_j}(x_j)$ for every $j\in\N$.
	\item For every sequence $\{\alpha_j\}_{j = 1}^\infty \subseteq \R$, we have
	\begin{equation}\label{prop:scaling_invariance:extremal_system}
		\Bigg \|\sum_{j = 1}^\infty \alpha_j u_j \Bigg\|_{L^{p^*, q}(\Omega)} \geq \frac{\lambda}{(1 + \varepsilon)^{\frac1{q}}} \Bigg( \sum_{j=1}^\infty |\alpha_j|^q \Bigg)^\frac1{q}.
	\end{equation}
\end{enumerate}
\end{proposition}
\begin{proof}
It is known that
\begin{equation}\label{prop:scaling_invariance:shrinking_property}
\|I\| = \|I_{G}\colon V_0^{m,p}(G) \to L^{p^*, q}(G)\| \qquad \text{for every open set $\emptyset\neq G\subseteq \Omega$}.
\end{equation}
Indeed, arguing as in the proof of \cite[Proposition~3.1]{LMOP:21}, we observe that the proof of \eqref{prop:scaling_invariance:shrinking_property} amounts to showing that, if $u\in V_0^{m, p}(B_r(0))$ and $0<s<r$, then
\begin{equation}\label{prop:scaling_invariance:shrinking_property2}
\frac{\|u(\kappa\, \cdot)\|_{L^{p^*,q}(B_s(0))}}{\|u(\kappa\, \cdot)\|_{V_0^{m, p}(B_s(0))}} = \frac{\|u\|_{L^{p^*,q}(B_r(0))}}{\|u\|_{V_0^{m, p}(B_r(0))}},
\end{equation}
where $\kappa = r/s$. It is a matter of simple straightforward computations to show that
\begin{align*}
\|u(\kappa\, \cdot)\|_{L^{p^*,q}(B_s(0))} &= \kappa^{m-\frac{d}{p}} \|u\|_{L^{p^*,q}(B_r(0))}
\intertext{and}
\|\nabla(u(\kappa\, \cdot))\|_{L^{p}(B_s(0))} &= \kappa^{m-\frac{d}{p}} \|\nabla u\|_{L^{p}(B_r(0))},
\end{align*}
whence \eqref{prop:scaling_invariance:shrinking_property2} immediately follows.

We now start with construction of the desired systems. We will use induction. First, using \eqref{prop:scaling_invariance:shrinking_property}, we find a ball $B_{r_1}(x_1)\subseteq \overline{B}_{r_1}(x_1)\subseteq \Omega$ and a function $u_1\in V_0^{m,p}(\Omega)$ such that $\spt u_1 \subseteq B_{r_1}(x_1)$, $\|u_1\|_{L^{p^*,q}(\Omega)} = \lambda$ and $\|u_1\|_{V_0^{m,p}(\Omega)} = 1$. Set $\delta_0 = |B_{r_1}(x_1)|$. By the monotone convergence theorem, there is $0 < \delta_1 < \delta_0$ such that
\begin{equation*}
(1 + \varepsilon) \int_{\delta_1}^{\delta_0} \left( t^{\frac1{p^*} - \frac1{q}} u_1^*(t) \right)^q \d{t} \geq \int_0^{\delta_0} \left( t^{\frac1{p^*} - \frac1{q}} u_1^*(t) \right)^q \d{t} = \|u_1\|_{L^{p^*,q}(\Omega)}^q.
\end{equation*}
Next, assume that we have already found functions $u_j \in V_0^{m,p}(\Omega)$, pairwise disjoint balls $B_{r_j}(x_j)\subseteq \overline{B}_{r_j}(x_j)\subseteq \Omega$, and $0 < \delta_k < \cdots < \delta_1 < \delta_0$, $j = 1, \dots, k$, where $k\in\N$, such that $\|u_j\|_{V_0^{m,p}(\Omega)} = 1$ and $\|u_j\|_{L^{p^*, q}(\Omega)} = \lambda$, $\spt u_j \subseteq B_{r_j}(x_j)$, and
\begin{equation}\label{prop:scaling_invariance:ineq}
(1 + \varepsilon) \int_{\delta_{j}}^{\delta_{j-1}} \left( t^{\frac1{p^*} - \frac1{q}} u_j^*(t) \right)^q \d{t} \geq \|u_j\|_{L^{p^*,q}(\Omega)}^q.
\end{equation}
Take any ball $B_{r_{k+1}}(x_{k+1})$ such that $B_{r_{k+1}}(x_{k+1}) \subseteq \overline{B}_{r_{k+1}}(x_{k+1}) \subseteq \Omega\setminus \bigcup_{j = 1}^k \overline{B}_{r_j}(x_j)$ and $|B_{r_{k+1}}(x_{k+1})| < \delta_k$. Thanks to \eqref{prop:scaling_invariance:shrinking_property}, we find a function $u_{k+1}\in V_0^{m,p}(\Omega)$ such that $\spt u_{k+1} \subseteq B_{r_{k+1}}(x_{k+1})$, $\|u_{k+1}\|_{L^{p^*,q}(\Omega)} = \lambda$ and $\|u_{k+1}\|_{V_0^{m,p}(\Omega)} = 1$. By the monotone convergence theorem again, there is $0 < \delta_{k+1} < \delta_{k}$ such that
\begin{equation*}
(1 + \varepsilon) \int_{\delta_{k+1}}^{\delta_k} \left( t^{\frac1{p^*} - \frac1{q}} u_{k+1}^*(t) \right)^q \d{t} \geq \int_0^{\delta_k} \left( t^{\frac1{p^*} - \frac1{q}} u_{k+1}^*(t) \right)^q \d{t} = \|u_{k+1}\|_{L^{p^*,q}(\Omega)}^q.
\end{equation*}
This finishes the inductive step.

Clearly, the constructed systems $\{u_j\}_{j=1}^\infty\subseteq V_0^{m,p}(\Omega)$ and $\{B_{r_j}(x_j)\}_{j=1}^\infty \subseteq \Omega$ have the properties (i)--(iii), and \eqref{prop:scaling_invariance:ineq} is valid for every $j\in\N$. It remains to verify that (iv) is also valid. Let $\{\alpha_j\}_{j = 1}^\infty \subseteq \R$. Since the functions $\{\alpha_j u_j\}_{j = 1}^\infty$ have pairwise disjoint supports, we have
\begin{equation*}
\left| \left\{x\in \Omega\colon \Bigg| \sum_{j = 1}^\infty \alpha_j u_j(x) \Bigg| > \gamma \right\} \right| = \sum_{j = 1}^\infty \left| \left\{x\in \Omega\colon \left| \alpha_j u_j(x) \right| > \gamma \right\} \right|
\end{equation*}
for every $\gamma > 0$. It follows that
\begin{equation}\label{prop:scaling_invariance:rearrangement_disjoint_supports}
\left( \sum_{j = 1}^\infty \alpha_j u_j \right)^* \geq \sum_{j = 1}^\infty |\alpha_j| u_j^* \chi_{(\delta_j, \delta_{j - 1})}.
\end{equation}
Indeed, suppose that there is $t\in(0, |\Omega|)$ such that
\begin{equation*}
\left( \sum_{j = 1}^\infty \alpha_j u_j \right)^*(t) < \sum_{j = 1}^\infty |\alpha_j| u_j^*(t) \chi_{(\delta_j, \delta_{j - 1})}(t).
\end{equation*}
Plainly, there is a unique index $k$ such that $t \in (\delta_k, \delta_{k - 1})$. By the definition of the nonincreasing rearrangement, there is $\gamma > 0$ such that
\begin{align*}
\left| \left\{x\in \Omega \colon \Bigg| \sum_{j = 1}^\infty \alpha_j u_j(x) \Bigg| > \gamma \right\} \right| \leq t \qquad \text{and} \qquad \gamma < |\alpha_k| u_k^*(t).
\end{align*}
Consequently, using the definition again, we have
\begin{equation*}
\left| \left\{x\in \Omega\colon \left|\alpha_k u_k(x) \right| > \gamma \right\} \right| > t,
\end{equation*}
however. Thus we have reached a contradiction, and so \eqref{prop:scaling_invariance:rearrangement_disjoint_supports} is proved.

Finally, using \eqref{prop:scaling_invariance:ineq} and \eqref{prop:scaling_invariance:rearrangement_disjoint_supports}, we observe that
\begin{align*}
\left \|\sum_{j = 1}^\infty \alpha_j u_j \right\|_{L^{p^*, q}(\Omega)}^q &= \int_0^\infty \left( t^{\frac1{p^*} - \frac1{q}} \left(\sum_{j = 1}^\infty \alpha_j u_j\right)^*(t) \right)^q \d{t} \\
&\geq \int_0^\infty \left( t^{\frac1{p^*} - \frac1{q}} \sum_{j = 1}^\infty |\alpha_j| u_j^*(t)\chi_{(\delta_{j}, \delta_{j-1})}(t) \right)^q \d{t} \\
&\geq \sum_{j = 1}^\infty |\alpha_j|^q \int_{\delta_j}^{\delta_{j-1}} \left( t^{\frac1{p^*} - \frac1{q}}  u_j^*(t) \right)^q \d{t} \\
&\geq \frac1{1 + \varepsilon} \sum_{j = 1}^\infty |\alpha_j|^q \|u_j\|_{L^{p^*,q}(\Omega)}^q \\
&= \frac{\lambda^q}{1 + \varepsilon} \sum_{j = 1}^\infty |\alpha_j|^q. \qedhere 
\end{align*}
\end{proof}

We start with the Lebesgue case. The following lemma of independent interest is a key ingredient for the proof of the fact that the embedding \eqref{prel:optimal_Leb} is finitely strictly singular. Its proof is inspired by that of \cite[Lemma~2.9]{BG:89}.

\begin{lemma}\label{lem:isometry_entropy_estimates}
Let $(R,\mu)$ be a probability measure space and $p\in[1, \infty)$. Let $X_n$ be a $n$-dimensional subspace of $L^p(R, \mu)$. There is a positive $\mu$-measurable function $g$ on $R$ and a linear isometry $L\colon L^p(R, \mu) \to L^p(R, \nu)$ defined as $Lf = g^{-1/p}f$, where $\d{\nu} = g\d{\mu}$, with the following properties. The measure $\nu$ is probabilistic, and in every subspace $Y\subseteq X_n$ with $\dim Y \geq n/2$, there exists a function $h\in Y$ such that 
\begin{align*}
\|h\|_{L^p(R,\mu)} &= 1 \\
\intertext{and}
\sup_{q\in[1, \infty)}\frac{\|Lh\|_{L^q(R,\nu)}}{\sqrt{q}} &\leq C.
\end{align*}
Here $C$ is an absolute constant depending only on $\min\{p, 2\}$.
\end{lemma}
\begin{proof}
Let $X_n$ be a $n$-dimensional subspace of $L^p(R, \mu)$. Thanks to \cite{L:78} (cf.~\cite[Theorem~2.1]{SZ:01}), there exists a positive $\mu$-measurable function $g$ on $R$ such that $\|g\|_{L^1(R,\mu)} = 1$ and the following is true: Upon setting $\dx\nu = g\d{\mu}$ and defining $Lf = g^{-1/p}f$, $f\in L^p(R,\mu)$, the subspace $\widetilde X_n = LX_n$ of $L^p(R, \nu)$ has a basis $\{\psi_1, \dots, \psi_n\}$ that is orthonormal in $L^2(R,\nu)$ and satisfies
\begin{equation}\label{lem:isometry_entropy_estimates:sum_of_Lewis_basis}
\sum_{j = 1}^n |\psi_j|^2 \equiv n \quad \text{$\mu$-a.e.~on $R$}.
\end{equation}
Note that, since $\widetilde X_n$ has a basis consisting of functions from $L^2(R,\nu)$, we have $\widetilde X_n \subseteq L^2(R,\nu)$ even for $p\in[1,2)$.

Let $Y$ be a subspace of $X_n$ with $\dim Y \geq n/2$. Set 
\begin{align*}
B_p(Z) &= \{f\in Z\colon \|f\|_{L^p(R,\nu)} \leq 1\}
\intertext{and}
B_{exp}(Z) &= \{f\in Z\colon \|f\|_{L^{\psi_2}(R,\nu)} \leq 1\},
\end{align*}
in which $Z$ is $\widetilde X_n$ or $\widetilde Y = LY$. By \cite[Lemma~9.2]{BLM:89}, we have
\begin{equation}\label{lem:isometry_entropy_estimates:entropy_L2-Orlicz}
\log E(B_2(\widetilde X_n), tB_{exp}(\widetilde X_n)) \leq c_1 \frac1{t^2} n \quad \text{for every $t\geq1$};
\end{equation}
here $c_1$ is an absolute constant, which is independent of $n$ and $t$. Since $\dim \widetilde Y \geq n/2$, we have
\begin{equation}\label{lem:isometry_entropy_estimates:volumetric_argument}
\log E(B_p(\widetilde Y), \frac1{4}B_p(\widetilde Y)) \geq n \log 2
\end{equation}
by a standard volumetric argument (e.g., see~\cite[(1.1.10)]{CS:90}).

We start with the case $p\in[2, \infty)$, which is simpler. Since $B_p(\widetilde Y)\subseteq B_2(\widetilde Y) \subseteq B_2(\widetilde X_n)$, it follows from \eqref{lem:isometry_entropy_estimates:entropy_L2-Orlicz} that
\begin{equation}\label{lem:isometry_entropy_estimates:entropy_Lp-Orlicz}
\log E(B_p(\widetilde Y), tB_{exp}(\widetilde X_n)) \leq c_1 \frac1{t^2} n \quad \text{for every $t\geq1$}.
\end{equation}
Moreover, since $E(B_p(\widetilde Y), 2tB_{exp}(\widetilde Y)) \leq E(B_p(\widetilde Y), tB_{exp}(\widetilde X_n))$ (e.g., see~\cite[Fact 4.1.9]{A-AGM:15}), we actually have
\begin{equation*}
\log E(B_p(\widetilde Y), 2tB_{exp}(\widetilde Y)) \leq c_1 \frac1{t^2} n \quad \text{for every $t\geq1$}.
\end{equation*}
Therefore, we can find $t_0\geq1$, not depending on $n$, so large that
\begin{equation*}
\log E(B_p(\widetilde Y), 2t_0B_{exp}(\widetilde Y)) \leq \frac{n\log 2}{2}.
\end{equation*}
It follows that $2t_0B_{exp}(\widetilde Y) \not\subseteq \frac1{4} B_p(\widetilde Y)$. Indeed, if $2t_0B_{exp}(\widetilde Y) \subseteq \frac1{4} B_p(\widetilde Y)$, then
\begin{equation*}
\log E(B_p(\widetilde Y), \frac1{4}B_p(\widetilde Y)) \leq \log E(B_p(\widetilde Y), 2t_0B_{exp}(\widetilde Y)) \leq \frac{n\log 2}{2},
\end{equation*}
which would contradict \eqref{lem:isometry_entropy_estimates:volumetric_argument}. Hence there is a function $h_0\in \widetilde Y$ such that
\begin{equation*}
\|h_0\|_{L^{\psi_2}(R,\nu)} \leq 2t_0 \qquad \text{and} \qquad \|h_0\|_{L^p(R, \nu)} > \frac1{4}.
\end{equation*}
Then $h=L^{-1}h_0/\|h_0\|_{L^p(R, \nu)}$ is the desired function thanks to \eqref{prel:expL2_extrapolation_norm}.

We now turn our attention to the case $p\in[1, 2)$. Assume for the moment that we know that
\begin{equation}\label{lem:isometry_entropy_estimates:entropy_Lp-L2}
\log E(B_p(\widetilde X_n), tB_2(\widetilde X_n)) \leq c_2 \frac{\log^2(1+t)}{t^2} n \quad \text{for every $t\geq1$}.
\end{equation}
Here $c_2$ is a constant depending only on $p$. Clearly
\begin{align*}
E(B_p(\widetilde Y), tB_{exp}(\widetilde X_n)) &\leq E(B_p(\widetilde X_n), tB_{exp}(\widetilde X_n))\\
&\leq E(B_p(\widetilde X_n), sB_2(\widetilde X_n)) \cdot E(sB_2(\widetilde X_n), tB_{exp}(\widetilde X_n)) \\
&= E(B_p(\widetilde X_n), sB_2(\widetilde X_n)) \cdot E(B_2(\widetilde X_n), \frac{t}{s}B_{exp}(\widetilde X_n)),
\end{align*}
and so
\begin{equation*}
\log E(B_p(\widetilde Y), tB_{exp}(\widetilde X_n)) \leq c_3 \left( \frac{\log^2(1+s)}{s^2} + \frac{s^2}{t^2} \right)n
\end{equation*}
for every $1\leq s\leq t$ thanks to \eqref{lem:isometry_entropy_estimates:entropy_L2-Orlicz} and \eqref{lem:isometry_entropy_estimates:entropy_Lp-L2}. Here $c_3$ is a constant depending only on $p$. Plugging $s = \sqrt{t}$ into this inequality, we arrive at
\begin{equation*}
\log E(B_p(\widetilde Y), tB_{exp}(\widetilde X_n)) \leq c_3 \frac{1 + \log^2(1+\sqrt{t})}{t} n \quad \text{for every $t\geq1$}.
\end{equation*}
Since $\lim_{t\to\infty}\frac{1 + \log^2(1+\sqrt{t})}{t} = 0$, we can now proceed in the same way as in the case $p\in[2, \infty)$, using this inequality instead of \eqref{lem:isometry_entropy_estimates:entropy_Lp-Orlicz}. Therefore, the proof will be complete once we prove \eqref{lem:isometry_entropy_estimates:entropy_Lp-L2}\textemdash to that end, we adapt the argument of \cite[Proposition~9.6]{BLM:89}. The proof of \eqref{lem:isometry_entropy_estimates:entropy_Lp-L2} is divided into several steps.

First, observe that
\begin{align}
&E(B_p(\widetilde X_n), tB_2(\widetilde X_n)) \label{lem:isometry_entropy_estimates:entropy_Lp-L2:est1}\\
&\quad\leq E(B_p(\widetilde X_n), 2B_p(\widetilde X_n)\cap 2tB_2(\widetilde X_n)) \cdot E(B_p(\widetilde X_n)\cap tB_2(\widetilde X_n), \frac{t}{2}B_2(\widetilde X_n)). \notag
\end{align}
Since $(\widetilde X_n, \|\cdot\|_{L^2(R,\nu)})$ is a Hilbert space, $2tB_2(\widetilde X_n)$ is a multiple of its unit ball and $B_p(\widetilde X_n)$ is a (nonempty) closed convex subset of $(\widetilde X_n, \|\cdot\|_{L^2(R,\nu)})$, we have $E(B_p(\widetilde X_n), 2tB_2(\widetilde X_n)) = \bar{E}(B_p(\widetilde X_n), 2tB_2(\widetilde X_n))$ (e.g., see~\cite[Fact~4.1.4]{A-AGM:15}). Consequently, if $B_p(\widetilde X_n)\subseteq\bigcup_{k=1}^m (u_k + 2tB_2(\widetilde X_n))$, where $u_k\in B_p(\widetilde X_n)$, then $B_p(\widetilde X_n)\subseteq \bigcup_{k=1}^m (u_k + 2B_p(\widetilde X_n)\cap 2tB_2(\widetilde X_n))$. Hence
\begin{equation}
E(B_p(\widetilde X_n), 2B_p(\widetilde X_n)\cap 2tB_2(\widetilde X_n)) \leq E(B_p(\widetilde X_n), 2tB_2(\widetilde X_n)). \label{lem:isometry_entropy_estimates:entropy_Lp-L2:est2}
\end{equation}
Combining \eqref{lem:isometry_entropy_estimates:entropy_Lp-L2:est1} and \eqref{lem:isometry_entropy_estimates:entropy_Lp-L2:est2}, we obtain
\begin{align}
&\log E(B_p(\widetilde X_n), tB_2(\widetilde X_n)) \notag\\
&\quad\leq \log E(B_p(\widetilde X_n), 2tB_2(\widetilde X_n)) \notag\\
&\quad\quad + \log E(B_p(\widetilde X_n)\cap tB_2(\widetilde X_n), \frac{t}{2}B_2(\widetilde X_n)). \label{lem:isometry_entropy_estimates:iteration_scheme}
\end{align}
Now, thanks to \eqref{lem:isometry_entropy_estimates:sum_of_Lewis_basis} and the fact that $p < 2$, we have
\begin{align*}
\|u\|_{L^2(R, \nu)}^2 &= \int_R \Big|\sum_{j = 1}^n \alpha_j \psi_j \Big|^2 \d{\nu} = \int_R \Big|\sum_{j = 1}^n \alpha_j \psi_j \Big|^{2 - p} \Big|\sum_{j = 1}^n \alpha_j \psi_j \Big|^p \d{\nu}\\
&\leq \Big(\sum_{j = 1}^n |\alpha_j|^2 \Big)^{\frac{2 - p}{2}} \int_R \Big(\sum_{j = 1}^n |\psi_j|^2 \Big)^{\frac{2 - p}{2}} \Big|\sum_{j = 1}^n \alpha_j \psi_j \Big|^p \d{\nu}\\
&= \|u\|_{L^2(R, \nu)}^{2 - p} n^{\frac{2 - p}{2}} \|u\|_{L^p(R, \nu)}^p
\end{align*}
for every $u = \sum_{j = 1}^n \alpha_j \psi_j \in \widetilde X_n$, whence it follows that
\begin{equation*}
\|u\|_{L^2(R, \nu)} \leq n^{\frac{2 - p}{2p}}  \|u\|_{L^p(R, \nu)} \quad \text{for every $u\in \widetilde X_n$}.
\end{equation*}
Clearly, this implies that $B_p(\widetilde X_n) \subseteq rB_2(\widetilde X_n)$ for every $r\geq n^{\frac{2 - p}{2p}}$; hence
\begin{equation*}
\log E(B_p(\widetilde X_n), 2^ktB_2(\widetilde X_n)) = 0
\end{equation*}
for every $k\in\N$ such that $2^kt\geq n^{\frac{2 - p}{2p}}$. Therefore, iterating \eqref{lem:isometry_entropy_estimates:iteration_scheme} with $t$ replaced by $2^kt$, $k\in\N$, we arrive at
\begin{equation}\label{lem:isometry_entropy_estimates:bound_on_entropy_Bp_to_tB2_sum}
\log E(B_p(\widetilde X_n), tB_2(\widetilde X_n)) \leq \sum_{k = 0}^\infty \log E(B_p(\widetilde X_n)\cap 2^ktB_2(\widetilde X_n), 2^{k-1}tB_2(\widetilde X_n)).
\end{equation}

Second, we claim that
\begin{equation}\label{lem:isometry_entropy_estimates:bound_on_entropy_Bp_to_sB2_desired}
\log E(B_p(\widetilde X_n)\cap sB_2(\widetilde X_n), \frac{s}{2}B_2(\widetilde X_n)) \leq c_4 \frac{\log^2(1+s)}{s^2}n
\end{equation}
for every $s\geq1$. Here $c_4$ is a constant depending only on $p$. Fix $s\geq1$. Let $Z$ denote $\widetilde{X}_n$ endowed with the norm
\begin{equation*}
\|u\|_{Z}=\max\left\{\|u\|_{L^p(R,\nu)}, \frac1{s}\|u\|_{L^2(R,\nu)}\right\}.
\end{equation*}
Note that $B_p(\widetilde X_n)\cap sB_2(\widetilde X_n)$ is the unit ball of $Z$. Owing to \cite[(9.1.7) together with Lemma~9.1.3]{A-AGM:15} combined with \cite[Lemma~4.4]{BLM:89}, we have
\begin{align}
&\log E(B_p(\widetilde X_n)\cap sB_2(\widetilde X_n), \frac{s}{2}B_2(\widetilde X_n)) \notag\\
&\quad \leq c_5 K(Z)^2C_2(Z)^2\pi_2^2(Z\hookrightarrow (\widetilde X_n, \|\cdot\|_{L^2(R, \nu)})) \frac1{s^2}. \label{lem:isometry_entropy_estimates:bound_on_entropy_Bp_to_sB2}
\end{align}
Here $c_5$ is an absolute constant.

As for $K(Z)$, we have
\begin{equation}\label{lem:isometry_entropy_estimates:pisier_ineq}
K(Z) \leq c_6 \log(1 + d(Z, \ell_2^n))
\end{equation}
by \eqref{prel:pisier_inequality}. Here $c_6$ is an absolute constant. 

We claim that $d(Z, \ell_2^n) \leq s$. To this end, consider the linear isomorphism $T\colon Z \to \ell_2^n$ defined as
\begin{equation*}
Tf = \{\alpha_j\}_{j = 1}^n,\ f = \sum_{j = 1}^n \alpha_j \psi_j \in Z.
\end{equation*}
Clearly, $T$ is onto $\ell_2^n$, and we have
\begin{equation}\label{lem:isometry_entropy_estimates:eq1}
\|Tf\|_{\ell_2^n} = \|f\|_{L^2(R, \nu)} \leq s \|f\|_{Z}
\end{equation}
for every $f\in Z$. On the other hand, using \eqref{lem:isometry_entropy_estimates:sum_of_Lewis_basis} and the fact that $p < 2$, we obtain
\begin{align*}
\|T^{-1}(\{\alpha_j\}_{j = 1}^n)\|_{L^p(R, \nu)}^p &= \int_R \left(\sum_{j = 1}^n \alpha_j \psi_j(t)\right)^p\d{\nu(t)} \\
&= \int_R \left(\sum_{j = 1}^n \alpha_j \psi_j(t)\right)^{p-2} \left(\sum_{j = 1}^n \alpha_j \psi_j(t)\right)^2\d{\nu(t)} \\
&\leq \int_R \left(\|\{\alpha_j\}_{j = 1}^n\|_{\ell_2^n}\sqrt{n}\right)^{p-2} \left(\sum_{j = 1}^n \alpha_j \psi_j(t)\right)^2\d{\nu(t)} \\
&\leq \|\{\alpha_j\}_{j = 1}^n\|_{\ell_2^n}^{p-2} \|\sum_{j = 1}^n \alpha_j \psi_j\|_{L^2(R,\nu)}^2 \\
& = \|\{\alpha_j\}_{j = 1}^n\|_{\ell_2^n}^p
\end{align*}
for every $\{\alpha_j\}_{j = 1}^n$. Furthermore, we plainly have
\begin{equation*}
\frac1{s}\|T^{-1}(\{\alpha_j\}_{j = 1}^n)\|_{L^2(R, \nu)} \leq \|T^{-1}(\{\alpha_j\}_{j = 1}^n)\|_{L^2(R, \nu)} = \|\{\alpha_j\}_{j = 1}^n\|_{\ell_2^n}.
\end{equation*}
Therefore $\|T^{-1}\| \leq 1$. By combining this with \eqref{lem:isometry_entropy_estimates:eq1}, it follows that
\begin{equation*}
d(Z, \ell_2^n) \leq s.
\end{equation*}
Plugging this into \eqref{lem:isometry_entropy_estimates:pisier_ineq}, we obtain
\begin{equation}\label{lem:isometry_entropy_estimates:convexity_constant_of_Z}
K(Z) \leq c_6 \log(1 + s).
\end{equation}

As for $C_2(Z)$, we claim that
\begin{equation}\label{lem:isometry_entropy_estimates:bound_on_2-cotype_constant}
C_2(Z) \leq c_7,
\end{equation}
where $c_7$ is a constant depending only on $p$. To this end, recall that
\begin{equation*}
\max\{C_2(L^p(R,\nu)), C_2(L^2(R,\nu)) \} < \infty,
\end{equation*}
and $C_2(L^p(R,\nu))$ depends only $p$ (e.g.,~see \cite[Theorem~6.2.14]{AK:16}). Now, since
\begin{align*}
\left(\sum_{j = 1}^m \|f_j\|_{Z}^2 \right)^{1/2} &\leq \left(\sum_{j = 1}^m \|f_j\|_{L^p(R,\nu)}^2 \right)^{1/2} + \frac1{s}\left(\sum_{j = 1}^m \|f_j\|_{L^2(R,\nu)}^2 \right)^{1/2} \\
&\leq 2 \max\big\{ C_2(L^p(R,\nu)), C_2(L^2(R,\nu)) \big\} \int_0^1 \Big\| \sum_{j = 1}^m f_j r_j(t) \Big\|_{Z} \d{t}
\end{align*}
for every $\{f_j\}_{j = 1}^m\subseteq Z$, $m\in\N$, the claim immediately follows.

As for $\pi_2(Z\hookrightarrow (\widetilde X_n, \|\cdot\|_{L^2(R, \nu)}))$, note that $\|f\|_{L^1(R,\nu)} \leq \|f\|_{L^p(R,\nu)} \leq \|f\|_{Z}$ for every $f\in Z$, and so the unit ball of $(\widetilde X_n, \|\cdot\|_{L^1(R, \nu)})^*$ is contained in the unit ball of $Z^*$. It follows that
\begin{equation*}
\pi_2(Z\hookrightarrow (\widetilde X_n, \|\cdot\|_{L^2(R, \nu)})) \leq \pi_2((\widetilde X_n, \|\cdot\|_{L^1(R, \nu)})\hookrightarrow (\widetilde X_n, \|\cdot\|_{L^2(R, \nu)})).
\end{equation*}
By \cite[Proof of Lemma~4.5]{BLM:89}, we have
\begin{equation*}
\pi_2((\widetilde X_n, \|\cdot\|_{L^1(R, \nu)})\hookrightarrow (\widetilde X_n, \|\cdot\|_{L^2(R, \nu)})) \leq c_8 \sqrt{n}.
\end{equation*}
Here $c_8$ is an absolute constant. Hence
\begin{equation}\label{lem:isometry_entropy_estimates:bound_on_2-summing_constant}
\pi_2(Z\hookrightarrow (\widetilde X_n, \|\cdot\|_{L^2(R, \nu)})) \leq c_8 \sqrt{n}.
\end{equation}

The desired estimate \eqref{lem:isometry_entropy_estimates:bound_on_entropy_Bp_to_sB2_desired} now follows by combining \eqref{lem:isometry_entropy_estimates:convexity_constant_of_Z}, \eqref{lem:isometry_entropy_estimates:bound_on_2-cotype_constant} and \eqref{lem:isometry_entropy_estimates:bound_on_2-summing_constant} with \eqref{lem:isometry_entropy_estimates:bound_on_entropy_Bp_to_sB2}.

Finally, now that we have \eqref{lem:isometry_entropy_estimates:bound_on_entropy_Bp_to_sB2_desired} at our disposal, the rest is simple. Combining \eqref{lem:isometry_entropy_estimates:bound_on_entropy_Bp_to_sB2_desired} with \eqref{lem:isometry_entropy_estimates:bound_on_entropy_Bp_to_tB2_sum}, we obtain
\begin{align*}
\log E(B_p(\widetilde X_n), tB_2(\widetilde X_n)) &\leq c_4 \sum_{k = 0}^\infty \frac{\log^2(1+2^kt)}{4^kt^2}n \\
&\leq 2c_4\left(\sum_{k = 0}^\infty \frac{k^2\log^22 + \log^2(1+t)}{4^k}\right) \frac1{t^2} n \\
&\leq 2c_4\left(\sum_{k = 0}^\infty \frac{k^2 + 1}{4^k}\right) \frac{\log^2(1+t)}{t^2} n.
\end{align*}
This finishes the proof of \eqref{lem:isometry_entropy_estimates:entropy_Lp-L2}.
\end{proof}

We are now in a position to prove the main result concerning the embedding \eqref{prel:optimal_Leb}.
\begin{theorem}\label{thm:optimal_Sobolev_Lebesgue}
Let $\Omega\subseteq\rd$ be a nonempty bounded open set, $m\in\N$, $1\leq m < d$, and $p\in [1, d/m)$. Denote by $I$ the identity operator $I\colon V_0^{m,p}(\Omega) \to L^{p^*}(\Omega)$, where $p^* = dp/(d-mp)$. There exists $n_0\in\N$, depending only on $d$ and $m$, such that
\begin{equation}\label{thm:optimal_Sobolev_Lebesgue:_estimate_on_bn}
C_1 n^{-\frac{m}{d}} \leq b_n(I) \leq C_2 n^{-\frac{m}{d}} \qquad \text{for every $n\geq n_0$}.
\end{equation}
Here $C_1$ and $C_2$ are constants depending only on $d$, $m$ and $p$.

In particular, $I$ is finitely strictly singular.
\end{theorem}
\begin{proof}
First, we prove the upper bound on $b_n(I)$. Set $l = d^m$. We may without loss of generality assume that $|\Omega|=1/l$; otherwise we replace $\dx x$ with $\dx x/(l|\Omega|)$. We start with a few definitions. By $G\colon V^{m, p}_0(\Omega) \to \bigoplus_{j = 1}^l L^p(\Omega)$ we denote the linear isometric operator defined as
\begin{equation*}
Gu = \nabla^m u,\ u\in V^{m, p}_0(\Omega).
\end{equation*}
Here $\bigoplus_{j = 1}^l$ stands for the $\ell_p$-direct sum, and the way in which the vector $\nabla^m u$ is ordered is completely immaterial\textemdash we fix arbitrary order. Furthermore, let $R = \bigoplus_{j = 1}^l \Omega^{(j)}$ consist of $l$ disjoint copies of $\Omega$, each endowed with the Lebesgue measure. We denote the corresponding probabilistic measure space by $(R, \mu)$. Finally, $S\colon \bigoplus_{j = 1}^l L^p(\Omega) \to L^p(R, \mu)$ denotes the linear isometry defined as
\begin{equation*}
S(f_1, \dots, f_l) = \sum_{j = 1}^l f_j\chi_{\Omega^{(j)}},\quad (f_1, \dots, f_l)\in \bigoplus_{j = 1}^l L^p(\Omega).
\end{equation*}

Let $c_1$ be the Besicovitch constant in $\rd$. Recall that $c_1$ depends only on $d$. Set $c_2 = \binom{d+m-1}{m-1}$. Note that $c_2$ is the dimension of the vector space of polynomials in $\rd$ of degree at most $m-1$, which we will denote by $\Pol_{m-1}(\rd)$. Assume that $n\geq 2 c_1 c_2$. Let $X_n$ be a $n$-dimensional subspace of $V_0^{m,p}(\Omega)$, and $\widetilde{X}_n\subseteq L^p(R, \mu)$ its image under the linear isometric operator $S\circ G$. Clearly, $\dim \widetilde{X}_n = n$. Let $L$, $g$ and $\nu$ be those from \myref{Lemma}{lem:isometry_entropy_estimates} applied to $\widetilde{X}_n$. Since $\Omega$ is bounded and $\|g\|_{L^1(R, \mu)} = 1$, for each $x\in \Omega$ we can find $r_x\in(0, \diam\Omega]$ such that
\begin{equation}\label{thm:optimal_Sobolev_Lebesgue:volume_of_balls}
\int_{\bigoplus_{j = 1}^l B^{(j)}_{r_x}(x)} g\d{\mu} = \frac{2c_1c_2}{n}.
\end{equation}
Here $B^{(j)}_{r_x}(x)$ are disjoint copies of $B_{r_x}(x)$ in $\Omega^{(j)}$. Using the Besicovitch covering lemma, we find a countable subcollection $\{B_{r_k}(x_k)\}_{k=1}^M$ such that
\begin{align}
\Omega \subseteq \bigcup_{k = 1}^M \bar{B}_{r_k}(x_k) \label{thm:optimal_Sobolev_Lebesgue:Besicovitch_cover}
\intertext{and}
\sum_{k = 1}^M\chi_{B_{r_k}(x_k)} \leq c_1. \label{thm:optimal_Sobolev_Lebesgue:Besicovitch_multiplicity}
\end{align}
We claim that
\begin{equation}\label{thm:optimal_Sobolev_Lebesgue:number_of_balls_upper}
M \leq \frac{n}{2c_2}.
\end{equation}
Indeed, we have
\begin{align*}
M \frac{2c_1c_2}{n} &= \sum_{k = 1}^M \int_{\bigoplus_{j = 1}^l B^{(j)}_{r_k}(x_k)} g\d{\mu}\\
&\leq \Big\| \sum_{k = 1}^M \chi_{\bigoplus_{j = 1}^l B^{(j)}_{r_k}(x_k)} \Big\|_{L^\infty(R,\mu)} \|g\|_{L^1(R,\mu)}\\
&\leq  c_1.
\end{align*}
Recall that, for every $u\in V^{m,p}(B_{r_k}(x_k))$, $k=1,\dots, M$, there is a polynomial $P_{u,k}\in\Pol_{m-1}(\rd)$, depending on $u$ and $B_{r_k}(x_k)$, such that
\begin{equation}\label{thm:optimal_Sobolev_Lebesgue:Poincare_Sobolev_balls}
\|u - P_{u,k}\|_{L^{p^*}(B_{r_k}(x_k))} \leq c_3 \|\nabla^m u\|_{L^p(B_{r_k}(x_k))};
\end{equation}
moreover, the dependence of $P_{u,k}$ on $u$ is linear. Here $c_3$ depends only on $d$, $m$ and $p$. This follows easily by iterating the classical Sobolev--Poincar\'e inequality on balls (e.g., see \cite[Corollary~1.64]{MZ:97}).

We claim that there is a subspace $Y$ of $X_n$ with $\dim Y\geq n/2$ such that
\begin{equation}\label{thm:optimal_Sobolev_Lebesgue:Poincare_Sobolev_on_subspace_codim_n/2}
\|u\|_{L^{p^*}(B_{r_k}(x_k))} \leq c_3 \|\nabla^m u\|_{L^p(B_{r_k}(x_k))} \quad \text{for every $u\in Y$ and $k$}.
\end{equation}
Indeed, set $Y_0 = X_n$, and let $Y_1$ be the kernel of the linear operator $Y_0 \ni u \mapsto P_{u, 1}\in\Pol_{m-1}(\rd)$. By the rank-nullity theorem, we have
\begin{equation*}
\dim Y_1 \geq n-\dim\Pol_{m-1}(\rd) = n - c_2.
\end{equation*}
Now, let $Y_2$ be the kernel of the linear operator $Y_1 \ni u \mapsto P_{u, 2}$. It follows that $\dim Y_2 \geq n-2c_2$. Proceeding in the obvious way, we find a subspace $Y=Y_M$ of $X_n$ with $\dim Y \geq n - Mc_2$ such that $P_{u, k}\equiv0$ for every $u \in Y$ and $k=1,\dots, M$. The claim now immediately follows from \eqref{thm:optimal_Sobolev_Lebesgue:number_of_balls_upper} and \eqref{thm:optimal_Sobolev_Lebesgue:Poincare_Sobolev_balls}.

Let $\widetilde{Y}$ be the image of $Y$ under the linear isometric operator $S\circ G$. Thanks to \myref{Lemma}{lem:isometry_entropy_estimates}, there is $\tilde{u}\in \widetilde{Y}\subseteq \widetilde{X}_n$ such that
\begin{align}
\|u\|_{V_0^{m, p}(\Omega)} = \|\tilde{u}\|_{L^p(R, \mu)} &= 1 \label{thm:optimal_Sobolev_Lebesgue:on_sphere}\\
\intertext{and}
\sup_{q\in[1, \infty)}\frac{\|L\tilde{u}\|_{L^q(R,\nu)}}{\sqrt{q}} &\leq c_4 \label{thm:optimal_Sobolev_Lebesgue:bounded_in_expL2},
\end{align}
where $u = (SG)^{-1}\widetilde{u}\in V_0^{m, p}(\Omega)$ and $\d{\nu} = g\d{\mu}$. Here $c_4$ is a constant depending only on $p$. By \eqref{thm:optimal_Sobolev_Lebesgue:Besicovitch_cover} and \eqref{thm:optimal_Sobolev_Lebesgue:Poincare_Sobolev_on_subspace_codim_n/2}, we have
\begin{align}
\|u\|_{L^{p^*}(\Omega)}^{p^*} &\leq \sum_{k=1}^M \|u\|_{L^{p^*}(B_{r_k}(x_k))}^{p^*} \leq c_3^{p^*} \sum_{k=1}^M \|\nabla^m u\|_{L^p(B_{r_k}(x_k))}^{p^*} \notag\\
&= c_3^{p^*} \sum_{k=1}^M \|\tilde{u}\chi_{\bigoplus_{j = 1}^l B^{(j)}_{r_k}(x_k)} \|_{L^p(R, \mu)}^{p^*} \notag\\
&= c_3^{p^*} \sum_{k=1}^M \|(L\tilde{u})\chi_{\bigoplus_{j = 1}^l B^{(j)}_{r_k}(x_k)} g^{1/p}\|_{L^p(R, \mu)}^{p^*} \notag\\
&= c_3^{p^*} \sum_{k=1}^M \|(L\tilde{u})\chi_{\bigoplus_{j = 1}^l B^{(j)}_{r_k}(x_k)}\|_{L^p(R, \nu)}^{p^*}. \label{thm:optimal_Sobolev_Lebesgue:final_chain_1}
\end{align}
Furthermore, by the H\"older inequality combined with the identity $1/p^* = 1/p - m/d$, \eqref{thm:optimal_Sobolev_Lebesgue:volume_of_balls}, \eqref{thm:optimal_Sobolev_Lebesgue:Besicovitch_cover} combined with \eqref{thm:optimal_Sobolev_Lebesgue:Besicovitch_multiplicity}, and \eqref{thm:optimal_Sobolev_Lebesgue:bounded_in_expL2}, we have
\begin{align}
&\sum_{k=1}^M \|(L\tilde{u})\chi_{\bigoplus_{j = 1}^l B^{(j)}_{r_k}(x_k)}\|_{L^p(R, \nu)}^{p^*} \notag\\
&\quad\leq \sum_{k=1}^M \|(L\tilde{u})\chi_{\bigoplus_{j = 1}^l B^{(j)}_{r_k}(x_k)}\|_{L^{p^*}(R, \nu)}^{p^*} \|\chi_{\bigoplus_{j = 1}^l B^{(j)}_{r_k}(x_k)}\|_{L^\frac{d}{m}(R, \nu)}^{p^*} \notag\\
&\quad= \sum_{k=1}^M \|(L\tilde{u})\chi_{\bigoplus_{j = 1}^l B^{(j)}_{r_k}(x_k)}\|_{L^{p^*}(R, \nu)}^{p^*} \|\chi_{\bigoplus_{j = 1}^l B^{(j)}_{r_k}(x_k)}\|_{L^1(R, \nu)}^{\frac{mp}{d-mp}} \notag\\
&\quad= \left(\frac{2 c_1 c_2}{n}\right)^{\frac{mp}{d-mp}} \sum_{k=1}^M \|(L\tilde{u})\chi_{\bigoplus_{j = 1}^l B^{(j)}_{r_k}(x_k)}\|_{L^{p^*}(R, \nu)}^{p^*} \notag\\
&\quad\leq c_1 \left(\frac{2 c_1 c_2}{n}\right)^{\frac{mp}{d-mp}}\|L\tilde{u}\|_{L^{p^*}(R, \nu)}^{p^*} \notag\\
&\quad\leq c_1 \left(\frac{2 c_1 c_2}{n}\right)^{\frac{mp}{d-mp}} c_4^{p^*}(p^*)^{p^*/2}. \label{thm:optimal_Sobolev_Lebesgue:final_chain_2}
\end{align}
Combining \eqref{thm:optimal_Sobolev_Lebesgue:final_chain_1} and \eqref{thm:optimal_Sobolev_Lebesgue:final_chain_2}, we obtain
\begin{equation}\label{thm:optimal_Sobolev_Lebesgue:final_ineq}
\|u\|_{L^{p^*}(\Omega)}^{p^*} \leq C_2^{p^*} n^{-{\frac{mp}{d-mp}}}.
\end{equation}
Here $C_2^{p^*} = c_1 (2 c_1 c_2)^{\frac{mp}{d-mp}} (c_3 c_4 \sqrt{p^*})^{p^*}$ depends only on $d$, $m$ and $p$. The desired upper bound in \eqref{thm:optimal_Sobolev_Lebesgue:_estimate_on_bn} now follows immediately from \eqref{thm:optimal_Sobolev_Lebesgue:on_sphere} and \eqref{thm:optimal_Sobolev_Lebesgue:final_ineq}.

Finally, we turn our attention to the lower bound in \eqref{thm:optimal_Sobolev_Lebesgue:_estimate_on_bn}, whose proof is simpler. To that end, recall that we have (e.g., see \cite[Remark~7]{P:04})
\begin{equation}\label{thm:optimal_Sobolev_Lebesgue:berns_lp->lp*}
b_n(\ell_p \to \ell_{p^*}) = n^{\frac{p - p^*}{pp^*}} = n^{-\frac{m}{d}} \quad \text{for every $n\in\N$}.
\end{equation}
Here we used the fact that $1\leq p < p^*$. Let $0 < \lambda < \|I\|$ and $\varepsilon > 0$. By \eqref{thm:optimal_Sobolev_Lebesgue:berns_lp->lp*}, there is a subspace $E_n$ of $\ell_p$ with $\dim E_n = n$ such that
\begin{equation}\label{thm:optimal_Sobolev_Lebesgue:berns_lp->lp*_almost_attained}
\inf_{\substack{\{\alpha_j\}_{j  = 1}^\infty \in E_n\\ \|\{\alpha_j\}_{j  = 1}^\infty\|_{\ell_p} = 1}} \|\{\alpha_j\}_{j  = 1}^\infty\|_{\ell_{p^*}} \geq n^{-\frac{m}{d}} - \varepsilon.
\end{equation}
Furthermore, let $\{u_j\}_{j = 1}^\infty$ and $\{B_j\}_{j = 1}^\infty$ be systems whose existence is guaranteed by \myref{Proposition}{prop:scaling_invariance} with $q = p^*$. Note that the linear operator $T\colon \ell_p \to V_0^{m, p}(\Omega)$ defined as
\begin{equation*}
T(\{\alpha_j\}_{j  = 1}^\infty) = \sum_{j = 1}^\infty \alpha_j u_j,\ \{\alpha_j\}_{j  = 1}^\infty \in \ell_p,
\end{equation*}
is well defined and isometric. Indeed, since the functions $u_j$ have mutually disjoint supports and $\|u_j\|_{V_0^{m, p}(\Omega)} = 1$, we have
\begin{equation*}
\Big\|\sum_{j = 1}^\infty \alpha_j u_j\Big\|_{V_0^{m, p}(\Omega)}^p = \Big\| \sum_{j = 1}^\infty \alpha_j \nabla u_j \Big\|_{L^p(\Omega)}^p =  \sum_{j = 1}^\infty |\alpha_j|^p \|\nabla u_j\|_{L^p(\Omega)}^p = \sum_{j = 1}^\infty |\alpha_j|^p.
\end{equation*}
In particular, $T$ is injective. Furthermore, we also have
\begin{equation}\label{thm:optimal_Sobolev_Lebesgue:Lp*_norm_=_lambda_ell_p*_norm}
\Big\| \sum_{j = 1}^\infty \alpha_j u_j \Big\|_{L^{p^*}(\Omega)} = \lambda \|\{\alpha_j\}_{j  = 1}^\infty\|_{\ell_{p^*}} \quad \text{for every $\{\alpha_j\}_{j  = 1}^\infty\in \ell_{p^*}$}
\end{equation}
since $\| u_j \|_{L^{p^*}(\Omega)} = \lambda$ for every $j\in\N$. Set $X_n = TE_n$. We have $\dim X_n = \dim E_n = n$. Combining \eqref{thm:optimal_Sobolev_Lebesgue:berns_lp->lp*_almost_attained} and \eqref{thm:optimal_Sobolev_Lebesgue:Lp*_norm_=_lambda_ell_p*_norm} with the fact that $T$ is isometric, we arrive at
\begin{align*}
b_n(I) &\geq \inf_{\substack{u\in X_n\\ \|u\|_{V_0^{m, p}(\Omega)} = 1}} \|u\|_{L^{p^*}(\Omega)} = \inf_{\substack{\{\alpha_j\}_{j  = 1}^\infty \in E_n\\ \|\{\alpha_j\}_{j  = 1}^\infty\|_{\ell_p} = 1}}  \Big\| \sum_{j = 1}^\infty \alpha_j u_j \Big\|_{L^{p^*}(\Omega)} \\
&= \lambda \inf_{\substack{\{\alpha_j\}_{j  = 1}^\infty \in E_n\\ \|\{\alpha_j\}_{j  = 1}^\infty\|_{\ell_p} = 1}}  \|\{\alpha_j\}_{j  = 1}^\infty\|_{\ell_{p^*}} \geq \lambda(n^{-\frac{m}{d}} - \varepsilon).
\end{align*}
Letting $\varepsilon\to 0^+$ and $\lambda \to \|I\|^-$, we obtain
\begin{equation*}
b_n(I) \geq \|I\|n^{-\frac{m}{d}}.
\end{equation*}
Note that this is actually the desired lower bound in \eqref{thm:optimal_Sobolev_Lebesgue:_estimate_on_bn} because we can take $C_1 = \|I\|$. Indeed, the norm of the embedding $V_0^{m,p}(\Omega) \to L^{p^*}(\Omega)$ depends only on $d$, $m$ and $p$ but not on $\Omega$. This follows from the simple observation that
\begin{align*}
\|I\colon V_0^{m,p}(B) \to L^{p^*}(B)\| &\leq \|I\colon V_0^{m,p}(\Omega) \to L^{p^*}(\Omega)\| \\
&\leq \|I\colon V_0^{m,p}(\tilde{B}) \to L^{p^*}(\tilde{B})\|,
\end{align*}
where $B$ and $\tilde{B}$ are (any) open balls in $\rd$ such that $B \subseteq \Omega \subseteq \tilde{B}$, and from the fact that $\|I\colon V_0^{m,p}(B) \to L^{p^*}(B)\|$ is constant for every open ball $B\subseteq\rd$ and depends only on $d, m$ and $p$\textemdash to that end, recall \eqref{prop:scaling_invariance:shrinking_property}.
\end{proof}

We conclude with the Lorentz case. The following theorem tells us that the ``really optimal'' Sobolev embedding \eqref{prel:optimal_Lor} is not strictly singular (let alone finitely strictly singular); moreover, all its Bernstein numbers coincide with its norm.

\begin{theorem}\label{thm:optimal_Sobolev_Lorentz}
Let $\Omega\subseteq\rd$ be a nonempty bounded open set, $m\in\N$, $1\leq m < d$, and $p\in[1, d/m)$. Denote by $I$ the identity operator $I\colon V_0^{m,p}(\Omega) \to L^{p^*, p}(\Omega)$, where $p^* = dp/(d-mp)$. We have
\begin{equation}\label{thm:optimal_Sobolev_Lorentz_NOT_SS:Bernstein_eq}
b_n(I) = \|I\| \qquad \text{for every $n\in\N$},
\end{equation}
where $\|I\|$ denotes the operator norm.

Furthermore, $I$ is not strictly singular.
\end{theorem}
\begin{proof}
Let $\varepsilon > 0$ and $0 < \lambda < \|I\|$, and $\{u_j\}_{j = 1}^\infty \subseteq V_0^{m, p}(\Omega)$ be a system of functions from \myref{Proposition}{prop:scaling_invariance} with $q = p$.

Thanks to the property (S1) of (strict) $s$-numbers, it is sufficient to show that
\begin{equation*}
b_n(I) \geq \|I\| \qquad \text{for every $n\in\N$}.
\end{equation*}
Let $X_n$ be the subspace of $V_0^{m, p}(\Omega)$ spanned by the functions $u_1, \dots, u_n$. Since the functions $u_j$ have mutually disjoint supports, it follows that $\dim X_n = n$. Since the functions $u_j$ have mutually disjoint supports, we have, for every $u = \sum_{j = 1}^n \alpha_j u_j \in X_n$,
\begin{equation*}
\|u\|_{V_0^{m, p}(\Omega)}^p = \sum_{j = 1}^n |\alpha_j|^p.
\end{equation*}
Furthermore, thanks to \eqref{prop:scaling_invariance:extremal_system},
\begin{equation*}
\|u\|_{L^{p^*, p}(\Omega)}^p = \Big\|\sum_{j = 1}^n \alpha_j u_j \Big\|_{L^p(\Omega)}^p \geq \frac{\lambda^p}{1 + \varepsilon} \sum_{j = 1}^n |\alpha_j|^p.
\end{equation*}
Hence
\begin{equation*}
b_n(I) \geq \inf_{u\in X_n\setminus\{0\}} \frac{\|u\|_{L^{p^*, p}(\Omega)}}{\|u\|_{V_0^{m, p}(\Omega)}} \geq \frac{\lambda}{(1 + \varepsilon)^{\frac1{p}}}.
\end{equation*}
Since this holds for every $\varepsilon > 0$ and $0 < \lambda < \|I\|$, it follows that $b_n(I)\geq \|I\|$.

Finally, to show that $I$ is not strictly singular, it is sufficient to take any $\varepsilon > 0$ and $0 < \lambda < \|I\|$ and consider the infinite dimensional subspace of $V_0^{m, p}(\Omega)$ spanned by the functions $u_1, u_2, \dots$ Arguing as above, we immediately see that $I$ is bounded from below on this infinite dimensional subspace. Therefore, $I$ is not strictly singular.
\end{proof}

\begin{remark}
In light of \eqref{prel:bernstein_numbers_smallest_injective}, \eqref{thm:optimal_Sobolev_Lorentz_NOT_SS:Bernstein_eq} actually tells us that, in the case of the ``really optimal'' Sobolev embedding \eqref{prel:optimal_Lor}, we have
\begin{equation*}
s_n(I) = \|I\|
\end{equation*}
for every $n\in\N$ and every injective strict $s$-number $s$.
\end{remark}


\end{document}